\documentclass[12pt]{amsart}

\usepackage{amsmath,amssymb,amsfonts,amsthm,latexsym,graphicx,multirow,hyperref,enumerate}

     \newcommand\Cay{\mathrm{Cay}}

\newtheorem{theorem}{Theorem}[section]
\newtheorem{lemma}[theorem]{Lemma}

\newtheorem{corollary}[theorem]{Corollary}

\theoremstyle{definition}

\def\Ga{\Gamma}

\def\ZZZ{\mathbb{Z}}

\newcommand{\pmat}[1]{\begin{pmatrix}#1\end{pmatrix}}

\usepackage{xcolor}
\usepackage[normalem]{ulem}
\newcommand{\cha}[2]{\textcolor{blue}{\sout{#1}}\textcolor{red}{#2}}

\begin{document}

\title[subgroup regular sets in Cayley graphs]{Subgroup regular sets in Cayley graphs}

\author[Wang]{Yanpeng Wang}
\thanks{Corresponding author: Yanpeng Wang}
\address{Rongcheng Campus\\Harbin University of Science and Technology\\Harbin, Heilongjiang 150080\\People’s Republic of China}
\email{wangyanpeng@pku.edu.cn}

\author[Xia]{Binzhou Xia}
\address{School of Mathematics and Statistics\\The University of Melbourne\\Parkville, VIC 3010\\Australia}
\email{binzhoux@unimelb.edu.au}

\author[Zhou]{Sanming Zhou}
\address{School of Mathematics and Statistics\\The University of Melbourne\\Parkville, VIC 3010\\Australia}
\email{sanming@unimelb.edu.au}

\date\today

\maketitle

\begin{abstract}
Let $\Gamma$ be a graph with vertex set $V$, and let $a$ and $b$ be nonnegative integers. A subset $C$ of $V$ is called an $(a,b)$-regular set in $\Gamma$ if every vertex in $C$ has exactly $a$ neighbors in $C$ and every vertex in $V\setminus C$ has exactly $b$ neighbors in $C$. In particular, $(0, 1)$-regular sets and $(1, 1)$-regular sets in $\Ga$ are called perfect codes and total perfect codes in $\Ga$, respectively. A subset $C$ of a group $G$ is said to be an $(a,b)$-regular set of $G$ if there exists a Cayley graph of $G$ which admits $C$ as an $(a,b)$-regular set. In this paper we prove that, for any generalized dihedral group $G$ or any group $G$ of order $4p$ or $pq$ for some primes $p$ and $q$, if a nontrivial subgroup $H$ of $G$ is a $(0, 1)$-regular set of $G$, then it must also be an $(a,b)$-regular set of $G$ for any $0\leqslant a\leqslant|H|-1$ and $0\leqslant b\leqslant |H|$ such that $a$ is even when $|H|$ is odd. A similar result involving $(1, 1)$-regular sets of such groups is also obtained in the paper.



\end{abstract}

\section{Introduction}

All groups considered in this paper are finite, and all graphs considered are finite, undirected and simple.
Let $\Gamma$ be a graph with vertex set $V$, and let $a, b$ be nonnegative integers. An \emph{$(a,b)$-regular set} \cite{Cardoso2019} in $\Gamma$ (or simply a \emph{regular set} in $\Gamma$ if the parameters $a, b$ are not important in the context) is a nonempty proper subset $C$ of $V$ such that $|\Gamma(v)\cap C|=a$ for each $v\in C$ and $|\Gamma(v)\cap C|=b$ for each $v \in V\setminus C$, where $\Gamma(v)$, the \emph{neighborhood} of $v$ in $\Gamma$, is the set of neighbors of $v$ in $\Ga$. (Two vertices are neighbors of each other if they are adjacent in the graph.) In particular, a $(0,1)$-regular set is called a \emph{perfect code}, and a $(1,1)$-regular set is called a \emph{total perfect code}.
In other words, a perfect code \cite{Big, Kratochvil1986} in $\Gamma$ is an independent set $C$ of $\Gamma$ such that every vertex in $V \setminus C$ has exactly one neighbor in $C$, and a total perfect code \cite{Zhou2016} in $\Gamma$ is a subset $C$ of $V$ such that every vertex of $\Gamma$ has exactly one neighbor in $C$.

We study regular sets in Cayley graphs in this paper. Given a group $G$ and an inverse-closed subset $S$ of $G\setminus\{e\}$, the \emph{Cayley graph} $\Cay(G,S)$ of $G$ with \emph{connection set} $S$ is the graph with vertex set $G$ such that $x,y\in G$ are adjacent if and only if $yx^{-1}\in S$. Herein and in the sequel we use $e$ to denote the identity element of the group under consideration. The main results in this paper affirmatively answer a question in \cite{WXZ2020} in the case when the group involved is a generalized dihedral group or a group of order $4p$ or $pq$, where $p$ and $q$ are primes. We prove that, for such a group $G$ and any nontrivial subgroup $H$ of $G$, if there exists a Cayley graph of $G$ which admits $H$ as a perfect code, then for any $0\leqslant a\leqslant|H|-1$ and $0\leqslant b\leqslant |H|$, with $a$ even when $|H|$ is odd, there exists a Cayley graph of $G$ which admits $H$ as an $(a,b)$-regular set.

Before formally stating our results, let us briefly discuss our motivation and related background. First, the study of perfect codes and total perfect codes is a significant part of the theory of domination in graphs \cite{HHS1998}, because a perfect code is exactly an efficient dominating set \cite{DS2003} (also known as an independent perfect dominating set~\cite{Lee2001}) and a total perfect code is exactly an efficient open dominating set \cite{HHS1998}. Second, as mentioned in \cite{HXZ2018, MWWZ2019}, perfect codes in Cayley graphs are especially interesting objects of study due to their connections with perfect codes in coding theory \cite{Va73}. In fact, the Hamming graph $H(n, q)$ and the Cartesian product $C_q^{\Box n}$ of $n$ copies of cycle $C_q$ with length $q$ are both Cayley graphs of $\ZZZ_q^n$. It is well known that the Hamming and Lee metrics over $\ZZZ_q^n$ are exactly the graph distances in $H(n, q)$ and $C_q^{\Box n}$, respectively. Therefore, perfect codes under these metrics \cite{Heden1, Va75} in classical coding theory are exactly perfect codes in Cayley graphs $H(n, q)$ and $C_q^{\Box n}$, respectively. So perfect codes in Cayley graphs can be considered as a generalization of perfect codes in the classical setting \cite{Va73}. For this reason, perfect codes in Cayley graphs have attracted considerable attention; see \cite[Section 1]{HXZ2018} for a brief account of results and \cite{CWX2020, HXZ2018, MWWZ2019, ZZ2021, ZZ2021a, Z15} for several recent studies in this line of research.

Thirdly, perfect codes in Cayley graphs are closely related to factorizations and tilings of groups. In general, a \emph{factorization} \cite{SS2009} of a group $G$ (into two factors) is a pair of subsets $(A, B)$ of $G$ such that every element of $G$ can be written uniquely as $ab$ with $a \in A$ and $b \in B$. If in addition $e \in A \cap B$, then such a factorization is called a \emph{tiling} \cite{Dinitz2006} of $G$. Beginning with G. Haj\'{o}s \cite{Hajos1942} in his proof of a well-known conjecture of Minkowski, the study of factorizations and tilings of groups is a classical topic which lies in the intersection of group theory and combinatorics. See, for example, \cite{Dinitz2006,RT1966,Szab2006} for several related results and \cite{SS2009} for a monograph on this topic. One can easily verify that $(A, B)$ is a tiling of $G$ such that $A$ is inverse-closed if and only if $B$ is a perfect code in $\Cay(G, A \setminus \{e\})$ with $e \in B$. Thus, results on perfect codes in Cayley graphs can be regarded as results on tilings of the underlying groups, and the converse is also true if the first factor is required to be inverse-closed. For example, the main result in \cite{RT1966} can be restated as follows: If $1 + (n(n-1)/2)$ is divisible by a prime exceeding $2 + \sqrt{n}$, then the complete transposition graph $T_n$ does not admit any perfect code, where $T_n$ is defined (see, for example, \cite{CLWZ2021}) as the Cayley graph of the symmetric group $S_n$ with connection set consisting of all transpositions in $S_n$.

Finally, a regular set in a regular graph is precisely one part of an equitable partition into two parts. More specifically, for a $k$-regular graph $\Ga$ with vertex set $V$, a subset $C$ of $V$ is an $(a, b)$-regular set in $\Ga$ if and only if $\{C, V\setminus C\}$ is an equitable partition of $\Ga$ with quotient matrix $\pmat{a & k-a\\b &k-b}$. In general, given a graph $\Ga$ with vertex set $V$, a partition $\mathcal{V} = \{V_1, V_2, \dots, V_r\}$ of $V$ is called an \emph{equitable partition} \cite[\S 9.3]{GR2001} (or a \emph{perfect colouring} \cite{F2007}) of $\Ga$ if there exists an $r\times r$ matrix $M=(m_{ij})$, called the \emph{quotient matrix} of $\mathcal{V}$, such that for any $i, j$ with $1 \le i, j \le r$, every vertex in $V_i$ has exactly $m_{ij}$ neighbors in $V_j$. It is known that all eigenvalues of $M$ are eigenvalues of $\Gamma$ (see \cite[Theorem 9.3.3]{GR2001}). In particular, if a regular graph admits an $(a, b)$-regular set, then it has $a-b$ as an eigenvalue. As one can find in \cite{G1993,GR2001}, equitable partitions play an important role in the study of many combinatorial structures, including distance-regular graphs and association schemes. Now assume that $\Gamma$ is a connected $k$-regular graph with vertex set $V$. Then $k$ is a simple eigenvalue of $M$~\cite[Theorem 9.3.3]{GR2001}, and the equitable partition $\mathcal{V}$ of $\Gamma$ is said to be \emph{$\mu$-equitable}~\cite{BCGG2019} if all eigenvalues of $M$ other than $k$ are equal to $\mu$. In particular, if an equitable partition $\{C, V\setminus C\}$ is $\mu$-equitable, then the nonempty proper subset $C$ of $V$ is called a \emph{$\mu$-perfect set} \cite{BCGG2019}. It was proved in \cite[Proposition 2.1]{BCGG2019} that, for a partition $\mathcal{V} = \{V_1, V_2, \dots, V_r\}$ of $V$, if $\mathcal{V}$ is $\mu$-equitable, then each $V_i$ is $\mu$-perfect, and conversely if $V_1, V_2, \dots, V_{r-1}$ are all $\mu$-perfect, then $\mathcal{V}$ is $\mu$-equitable. Thus it is particularly important to study equitable partitions with exactly two parts. Such two-part equitable partitions are called \emph{perfect $2$-colourings} \cite{F2007}, and they are essentially regular sets as seen above. Perfect $2$-colourings are closely related to coding theory and as such have been studied extensively over many years. See, for example, \cite{GG2013} for a study of perfect $2$-colorings of Johnson graphs $J(v,3)$, and \cite{BKMV2021, MV2020} for some recent results on perfect $2$-colourings of Hamming graphs. In \cite{BCGG2019}, equitable partitions of Latin square graphs are studied and those whose quotient matrix does not have an eigenvalue $-3$ are classified. In \cite{RCZ2018}, a few results on equitable partitions and regular sets of Cayley graphs involving irreducible characters of the underlying groups are obtained.


A subset $C$ of a group $G$ is called \cite{HXZ2018} a \emph{(total) perfect code of $G$} if it is a (total) perfect code in some Cayley graph of $G$. In general, a subset $C$ of $G$ is called an \emph{$(a,b)$-regular set of $G$} if $C$ is an $(a,b)$-regular set in some Cayley graph of $G$. As noted in \cite{HXZ2018}, subgroup perfect codes (that is, subgroups which are perfect codes of the group under consideration) are particularly interesting since they are an analogue of perfect linear codes \cite{Heden1, Va75} in coding theory. The study of subgroup perfect codes was initiated in \cite{HXZ2018}, and further results on them were obtained in \cite{B21, ZZ2021, ZZ2021a}. Recently, a characterization of those groups whose subgroups are all perfect codes of the group was given in \cite{MWWZ2019}.

In \cite[Theorem~2.2]{HXZ2018}, it was proved that, for a normal subgroup $H$ of a group $G$, $H$ is a perfect code of $G$ if and only if
\begin{equation}\label{equ9}
\text{for any $g\in G$ with $g^2\in H$, there exists $h\in H$ such that $(gh)^2=e$},
\end{equation}
and $H$ is a total perfect code of $G$ if and only if~\eqref{equ9} holds and $|H|$ is even. In \cite{WXZ2020}, the authors of the present paper improved this result as follows.

\begin{theorem}\label{thm6}
\emph{(\cite[Theorem~1.2]{WXZ2020})}
Let $G$ be a group and let $H$ be a nontrivial normal subgroup of $G$. Then the following statements are equivalent:
\begin{enumerate}
\item[{\rm(a1)}] $G$ and $H$ satisfy condition \eqref{equ9};
\item[{\rm(a2)}] $H$ is a perfect code of $G$;
\item[{\rm(a3)}] $H$ is an $(a,b)$-regular set of $G$ for every pair of integers $a$ and $b$ with $0\leqslant a\leqslant|H|-1$ and $0\leqslant b\leqslant |H|$ such that $\gcd(2,|H|-1)$ divides $a$.
\end{enumerate}
And the following statements are also equivalent:
\begin{enumerate}
\item[{\rm(b1)}] $G$ and $H$ satisfy condition \eqref{equ9}, and $|H|$ is even;
\item[{\rm(b2)}] $H$ is a total perfect code of $G$;
\item[{\rm(b3)}] $H$ is an $(a,b)$-regular set of $G$ for every pair of integers $a$ and $b$ with $0\leqslant a\leqslant|H|-1$ and $0\leqslant b\leqslant |H|$.
\end{enumerate}
\end{theorem}

Since (a3) implies (a2) and (b3) implies (b2), Theorem \ref{thm6} essentially says that (a2) implies (a3) and (b2) implies (b3). In \cite[Question~1.3]{WXZ2020}, we asked whether these implications are still true if the subgroup $H$ of $G$ is not normal. In this paper we give an affirmative answer to this question in the case when $G$ is a generalized dihedral group or a group of order $4p$ or $pq$ for some primes $p$ and $q$. A group $G$ is called a \emph{generalized dihedral group} with respect to $A$ if $A$ is an abelian subgroup of $G$ and there exists an element $y\in G$ such that
\[
G=\langle A,y\mid y^2=e,x^y=x^{-1},\forall x\in A\rangle,
\]
where $x^y = y^{-1}xy$.

The main results in this paper are the following two theorems.

\begin{theorem}\label{thm4}
Let $G$ be a generalized dihedral group and let $H$ be a nontrivial subgroup of $G$. Then the following hold:
\begin{enumerate}[\rm (a)]
\item $H$ is a perfect code of $G$ if and only if $H$ is an $(a,b)$-regular set of $G$ for every pair of integers $a, b$ with $0\leqslant a\leqslant|H|-1$, $0\leqslant b\leqslant |H|$, and $a$ even when $|H|$ is odd;
\item $H$ is a total perfect code of $G$ if and only if $H$ is an $(a,b)$-regular set of $G$ for every pair of integers $a, b$ with $0\leqslant a\leqslant|H|-1$ and $0\leqslant b\leqslant |H|$.
\end{enumerate}
\end{theorem}


\begin{theorem}\label{thm1}
Let $G$ be a group of order $4p$ or $pq$ for some primes $p$ and $q$, and let $H$ be a nontrivial subgroup of $G$. Then the following hold:
\begin{enumerate}[\rm (a)]
\item $H$ is a perfect code of $G$ if and only if $H$ is an $(a,b)$-regular set of $G$ for every pair of integers $a, b$ with $0\leqslant a\leqslant|H|-1$, $0\leqslant b\leqslant |H|$, and $a$ even when $|H|$ is odd;
\item $H$ is a total perfect code of $G$ if and only if $H$ is an $(a,b)$-regular set of $G$ for every pair of integers $a, b$ with $0\leqslant a\leqslant|H|-1$ and $0\leqslant b\leqslant |H|$.
\end{enumerate}
\end{theorem}


As mentioned earlier, every regular graph admitting an $(a,b)$-regular set has $a-b$ as one of its eigenvalues. Obviously, $\{a-b: 0\leqslant a\leqslant|H|-1, 0\leqslant b\leqslant |H|\} = \{-|H|, -|H|+1, \dots, |H|-1\}$, and also $\{a-b: 0\leqslant a\leqslant|H|-1, 0\leqslant b\leqslant |H|, \mbox{$a$ is even}\} = \{-|H|, -|H|+1, \dots, |H|-1\}$ when $|H|$ is odd. So part (a) in Theorems \ref{thm6} and \ref{thm4} implies the following result.

\begin{corollary}
Let $G$ be a generalized dihedral group or a group of order $4p$ or $pq$ for some primes $p$ and $q$. Let $H$ be a nontrivial subgroup of $G$. If $H$ is a perfect code of $G$, then for every integer $\ell$ between $-|H|$ and $|H|-1$ there exists a Cayley graph of $G$ which has $\ell$ as one of its eigenvalues.
\end{corollary}

A few comments are in order. As we will see in Lemma \ref{lem4}, to prove that (a2) implies (a3) for a group-subgroup pair $(G, H)$ a crucial step is to prove that for each integer $b$ between $0$ and $|H|$ there exists an inverse-closed subset of $G\setminus H$ which is the union of $b$ pairwise disjoint right transversals of $H$ in $G\setminus H$. This turns out to be challenging for a general group-subgroup pair $(G, H)$. However, when $G$ is a generalized dihedral group or a group of order $4p$ or $pq$, we manage to construct such a subset using the fact that in these cases the number of right cosets of $H$ in $G$ contained in each double coset of $H$ in $G$ is small or a prime number. It would be interesting to study for which families of groups the statements in Theorems \ref{thm4} and \ref{thm1} hold. At the time of writing we are not aware of any group for which these statements are not true. In fact, we have manually checked all groups of order up to $40$ and found that these statements hold for all of them.

The rest of the paper is structured as follows. In Section~\ref{sec1}, we present characterizations of perfect codes and regular sets in Cayley graphs. Based on these characterizations we establish several technical lemmas in Section~\ref{sec2}. Using these preparatory results, proofs of Theorem~\ref{thm4} and Theorem~\ref{thm1} will be given in Section~\ref{sec3}.

\section{Characterizations of perfect codes and regular sets}\label{sec1}

Let $G$ be a group and $H$ a subgroup of $G$. Denote by $[G{:}H]$ the set of right cosets of $H$ in $G$, and define a binary relation $\sim$ on $[G{:}H]$ such that
\[
Hx\sim Hy\Leftrightarrow y\in Hx^{-1}H.
\]
It is readily seen that the relation $\sim$ is well defined and symmetric. Define $\Gamma$ to be the graph with vertex set $[G{:}H]$ such that $\{Hx,Hy\}$ is an edge if and only if $Hx\neq Hy$ and $Hx\sim Hy$.
Then for any $x\in G$ and $h\in H$,
\begin{equation}\label{eq1}
Hxh\sim Hy\text{ for all $y\in x^{-1}H$}\quad\text{and}\quad Hx^{-1}h\sim Hz\text{ for all $z\in xH$}.
\end{equation}
Since each double coset of $H$ is a union of right cosets of $H$ in $G$, we may view $HxH$ and $Hx^{-1}H$ as sets of vertices of $\Gamma$. Then~\eqref{eq1} shows that the subgraph of $\Gamma$ induced by $HxH\cup Hx^{-1}H$ is a connected component of $\Gamma$, which we denote by $\Gamma_x$.
Note that in $\Gamma_x$ each edge induces a pair of inverse elements. Note also that $\Gamma$ and $\Gamma_x$ depend on $G$ and $H$, but in our subsequently discussion the groups underlying $\Gamma$ and $\Gamma_x$ should be clear from the context.

The following lemma gives the structure of $\Gamma_x$, where $H^x = x^{-1}Hx$.

\begin{lemma}\label{lem5}\cite{CWX2020}
Let $G$ be a group and let $H$ be a subgroup of $G$. Let $\Gamma_x$ be as above and let $m=|H|/|H\cap H^x|$. If $HxH=Hx^{-1}H$, then $\Gamma_x$ is the complete graph $K_m$. If $HxH\neq Hx^{-1}H$, then $\Gamma_x$ is the complete bipartite graph $K_{m,m}$. 
\end{lemma}

\begin{theorem}\label{thm2}\cite{CWX2020}
Let $G$ be a group and let $H$ be a subgroup of $G$.
Then the following statements are equivalent:
\begin{enumerate}[{\rm(a)}]
\item $H$ is a perfect code of $G$;
\item there exists an inverse-closed right transversal of $H$ in $G$;
\item for each $x\in G$ such that $x^2\in H$ and $|H|/|H\cap H^x|$ is odd, there exists $y\in Hx$ such that $y^2=e$;
\item for each $x\in G$ such that $HxH=Hx^{-1}H$ and $|H|/|H\cap H^x|$ is odd, there exists $y\in Hx$ such that $y^2=e$;
\item for each $x\in G\setminus H$ such that $x^2\in H$ and $|H|/|H\cap H^x|$ is odd, there exists an involution in $Hx$;
\item for each $x\in G\setminus H$ such that $HxH=Hx^{-1}H$ and $|H|/|H\cap H^x|$ is odd, there exists an involution in $Hx$.
\end{enumerate}
\end{theorem}

As usual, for a group $G$, denote by $\mathbb{Z}[G]$ the group ring of $G$ over $\mathbb{Z}$. For a subset $A$ of group $G$, denote
\[
\overline{A}=\sum_{g\in G}\mu_A(g)g\in \mathbb{Z}[G],
\]
where
\[
\mu_A(g)=\left\{\begin{aligned}
1,&\quad g\in A;\\
0,&\quad g\in G\setminus A.
\end{aligned}
\right.
\]
Since $\mathbb{Z}[G]$ is a ring, for any subsets $A, B$ of $G$, the product $\overline{A} \cdot \overline{B}$ is well defined and is equal to $\sum_{g, h \in G} (\mu_A(g) \mu_B(h)) gh$.

\begin{lemma}\label{lem4}  
Let $G$ be a group, $H$ a subgroup of $G$, and $S$ an inverse-closed subset of $G\setminus\{e\}$. Let $a$ and $b$ be nonnegative integers. Suppose that $H$ is a perfect code in some Cayley graph $\Cay(G,S_0)$ of $G$. Then $H$ is an $(a,b)$-regular set in $\Cay(G,S)$ if and only if $|S\cap H|=a$ and $\overline{S\setminus H}\cdot \overline{H}=b\,\overline{S_0}\cdot \overline{H}$.
\end{lemma}

\begin{proof}
Since $H$ is a perfect code in $\Cay(G,S_0)$, we have $\overline{G}=\overline{S_0\cup\{e\}}\cdot\overline{H}=\overline{S_0}\cdot\overline{H}+\overline{H}$. Hence $\overline{G\setminus H}=\overline{G}-\overline{H}=\overline{S_0}\cdot\overline{H}$. On the other hand, $H$ is an $(a,b)$-regular set in $\Cay(G,S)$ if and only if $\overline{S}\cdot\overline{H}=a\,\overline{H}+b\,\overline{G\setminus H}$, which holds if and only if $|S\cap H|=a$ and $\overline{S\setminus H}\cdot \overline{H}=b\,\overline{G\setminus H}$ as $S$ is the disjoint union of $S\cap H$ and $S\setminus H$ and $\overline{h}\cdot \overline{H}=\overline{H}$ for any $h\in H$. Therefore, $H$ is an $(a,b)$-regular set in $\Cay(G,S)$ if and only if $|S\cap H|=a$ and $\overline{S\setminus H}\cdot\overline{H}=b\,\overline{S_0}\cdot\overline{H}$.
\end{proof}

Let $G$ be a group, let $H$ be a subgroup of $G$, and let $K$ be a set of right cosets of $H$ in $G$. A \emph{transversal of $H$ in $K$} is a subset of $G$ which is formed by taking exactly one element from each right coset of $H$ in $K$. By Lemma~\ref{lem4}, to prove that $H$ is an $(a,b)$-regular set of $G$ it suffices to prove the existence of an inverse-closed subset $S_a$ of $H\setminus\{e\}$ with size $a$ and an inverse-closed subset $S_b$ of $G\setminus H$ which is the union of $b$ pairwise disjoint right transversals of $H$ in $G\setminus H$. In fact, Lemma~\ref{lem4} ensures that $H$ is an $(a,b)$-regular set in $\Cay(G, S)$, where $S = S_a \cup S_b$.

\section{Technical lemmas}\label{sec2}

\begin{lemma}\label{lem2}
Let $G$ be a group, let $H$ be a nontrivial subgroup of $G$, and let $x\in G$. If $HxH\neq Hx^{-1}H$, then there is no involution in $HxH\cup Hx^{-1}H$. If there is an involution in $HxH$, then there is an involution in each right coset in $HxH$.
\end{lemma}

\begin{proof}
First assume that $HxH\neq Hx^{-1}H$. Then $HxH\cap Hx^{-1}H=\emptyset$. Suppose that $y$ is an involution in $HxH$. Then $y^{-1}\in (HxH)^{-1}=Hx^{-1}H$, which is a contradiction.

Next assume that there is an involution in $HxH$, say, $y$. Then for each $h\in H$, the element $h^{-1}yh$ is an involution in $Hxh$. This shows that each right coset in $HxH$ contains an involution.
\end{proof}

To obtain transversals of $H$ in $G\setminus H$ it suffices to find those of $H$ in $HxH\cup Hx^{-1}H$ for each $\Gamma_x$ with $x\in G$. For this purpose, we will analyze the cases $HxH=Hx^{-1}H$ and $HxH\neq Hx^{-1}H$ separately. First, we consider the case $HxH\neq Hx^{-1}H$.

\begin{lemma}\label{lem9}
Let $G$ be a group, let $H$ be a nontrivial subgroup of $G$, and let $x\in G$. Suppose that $HxH\neq Hx^{-1}H$ and $|HxH|/|H|=1$ or $2$. Then there exist $|H|$ pairwise disjoint inverse-closed right transversals of $H$ in $HxH\cup Hx^{-1}H$. In particular, for any integer $0\leqslant b\leqslant |H|$, there exist $b$ pairwise disjoint right transversals of $H$ in $HxH\cup Hx^{-1}H$ whose union is inverse-closed.
\end{lemma}

\begin{proof}
Since $HxH\neq Hx^{-1}H$, we have $HxH\cap Hx^{-1}H=\emptyset$.

First suppose that $|HxH|/|H|=1$. Then each of $HxH$ and $Hx^{-1}H$ consists of a single right coset of $H$. For each $r\in HxH$, since $r^{-1}\in Hx^{-1}H$, we see that $\{r,r^{-1}\}$ is a right transversal of $H$ in $HxH\cup Hx^{-1}H$. Write $HxH=\{r_1,\dots,r_{|H|}\}$. Then
\[
\{r_1,r_1^{-1}\},\dots,\{r_{|H|},r_{|H|}^{-1}\}
\]
are $|H|$ pairwise disjoint inverse-closed right transversals of $H$ in $HxH\cup Hx^{-1}H$.

Next suppose that $|HxH|/|H|=2$. Then each of $HxH$ and $Hx^{-1}H$ consists of two right cosets of $H$. Let the right cosets of $H$ in $HxH$ be $Hx_1$ and $Hx_2$, and let the right cosets of $H$ in $Hx^{-1}H$ be $Hy_1$ and $Hy_2$. Take
\begin{align*}
S=\{z\mid z\in Hx_1,\,z^{-1}\in Hy_1\},\quad &T=\{z\mid z\in Hx_2,\,z^{-1}\in Hy_2\},\\
U=\{z\mid z\in Hx_1,\,z^{-1}\in Hy_2\},\quad &V=\{z\mid z\in Hx_2,\,z^{-1}\in Hy_1\}.
\end{align*}
Then $S$, $T$, $U$, $V$, $S^{-1}$, $T^{-1}$, $U^{-1}$, $V^{-1}$ are pairwise disjoint, and we have
\[
Hx_1=S\cup U,\ \ Hx_2=T\cup V,\ \ Hy_1=S^{-1}\cup V^{-1},\ \ Hy_2=T^{-1}\cup U^{-1}.
\]
Hence $|H|=|S|+|U|=|T|+|V|=|S|+|V|=|T|+|U|$, which implies that $|S|=|T|$ and $|U|=|V|$. Write $S=\{s_1,\dots,s_c\}$, $T=\{t_1,\dots,t_c\}$, $U=\{u_1,\dots,u_d\}$ and $V=\{v_1,\dots,v_d\}$, where $c+d = |H|$. Then for each $i\in\{1,\dots,c\}$ and $j\in\{1,\dots,d\}$, the sets $\{s_i,t_i,s_i^{-1},t_i^{-1}\}$ and $\{u_j,v_j,u_j^{-1},v_j^{-1}\}$ are both right transversals of $H$ in $HxH\cup Hx^{-1}H=Hx_1\cup Hx_2\cup Hy_1\cup Hy_2$. Therefore,
$$\{s_1,t_1,s_1^{-1},t_1^{-1}\}, \dots, \{s_c,t_c,s_c^{-1},t_c^{-1}\},
\{u_d,v_d,u_d^{-1},v_d^{-1}\}, \dots, \{u_d,v_d,u_d^{-1},v_d^{-1}\}$$
are $|H|$ pairwise disjoint inverse-closed right transversals of $H$ in $HxH\cup Hx^{-1}H$.
\end{proof}



\begin{lemma}\label{lem10}
Let $G$ be a group, let $H$ be a nontrivial subgroup of $G$, and let $x\in G$. Suppose that $HxH\neq Hx^{-1}H$ and $|H|$ is a prime. Then there exist $|H|$ pairwise disjoint inverse-closed right transversals of $H$ in $HxH\cup Hx^{-1}H$. In particular, for any integer $0\leqslant b\leqslant |H|$, there exist $b$ pairwise disjoint right transversals of $H$ in $HxH\cup Hx^{-1}H$ whose union is inverse-closed.
\end{lemma}

\begin{proof}
Since $|HxH|/|H|=|H|/|H\cap H^x|$ and $|H|$ is a prime, we have $|HxH|/|H|=1$ or $|H|$. If $|HxH|/|H|=1$, then the conclusion follows from Lemma~\ref{lem9}. Now assume that $|HxH|/|H|=|H|$. Then Lemma~\ref{lem5} asserts that $\Gamma_x$ is a complete bipartite regular graph of valency $|H|$. By K\"{o}nig's $1$-factorization theorem \cite{Konig1976} (see also \cite[Corollary 16.6]{BM}), the $|H|$-regular bipartite graph $\Gamma_x$ can be decomposed into $|H|$ edge-disjoint perfect matchings. These $|H|$ perfect matchings induce $|H|$ pairwise disjoint right transversals of $H$ in $HxH\cup Hx^{-1}H$. Moreover, by the definition of $\Gamma_x$, each of these right transversals is inverse-closed. This completes the proof.
\end{proof}


Next we consider the case $HxH=Hx^{-1}H$.

\begin{lemma}\label{lem11}
Let $G$ be a group, let $H$ be a nontrivial subgroup of $G$, and let $x\in G$. Suppose that $HxH=Hx^{-1}H=Hx$ and there is an involution in $Hx$. Then for any integer $0\leqslant b\leqslant |H|$, there exist $b$ pairwise disjoint right transversals of $H$ in $HxH$ whose union is inverse-closed.
\end{lemma}

\begin{proof}
In this case, every single element in $Hx$ gives a right transversal of $H$ in $HxH\cup Hx^{-1}H=Hx$. Since $Hx$ is inverse-closed and contains at least one involution, we can write
\[
Hx=\{r_1,\dots,r_c,s_1,s^{-1}_1,\dots,s_d,s^{-1}_d\}
\]
for some integers $c, d$ with $1\leqslant c\leqslant|H|$ and $c+2d=|H|$, where $r_i$ is an involution for $i=1,\dots,c$ and $s_j$ has order greater than $2$ for $j=1,\dots,d$. If $b>2d$, then $0<b-2d\leqslant|H|-2d=c$, and we take
\[
R=\{r_1,\dots,r_{b-2d},s_1,s^{-1}_1,\dots,s_d,s^{-1}_d\}.
\]
If $b\leqslant2d$ and $b$ is odd, then we take
\[
R=\{r_1,s_1,s^{-1}_1,\dots,s_{(b-1)/2},s^{-1}_{(b-1)/2}\}.
\]
If $b\leqslant2d$ and $b$ is even, then we take
\[
R=\{s_1,s^{-1}_1,\dots,s_{b/2},s^{-1}_{b/2}\}.
\]
In each case $R$ consists of $b$ pairwise disjoint transversals of $H$ in $HxH\cup Hx^{-1}H$, and moreover $R$ is inverse-closed.
\end{proof}

\begin{lemma}\label{lem12}
Let $G$ be a group, let $H$ be a nontrivial subgroup of $G$, and let $x\in G$. Suppose that $HxH=Hx^{-1}H$ and $|HxH|/|H|=2$. Then there exist $|H|$ pairwise disjoint inverse-closed right transversals of $H$ in $HxH$. In particular, for any integer $0\leqslant b\leqslant |H|$, there exist $b$ pairwise disjoint right transversals of $H$ in $HxH$ whose union is inverse-closed.
\end{lemma}

\begin{proof}
Since $|HxH|/|H|=2$, we have $HxH=Hx\cup Hy$ for some $y\in G$. Take
\begin{align*}
S=\{z\mid z\in Hx,\,z^{-1}\in Hy\},\quad &T=\{z\mid z\in Hy,\,z^{-1}\in Hx\},\\
U=\{z\mid z\in Hx,\,z^2=e\},\quad &V=\{z\mid z\in Hy,\,z^2=e\},\\
X=\{z\mid z\in Hx,\,z^{-1}\in Hx,\,z^2\ne e\},\quad &Y=\{z\mid z\in Hy,\,z^{-1}\in Hy,\,z^2\ne e\}.
\end{align*}
By Lemma~\ref{lem5}\cha{}{,} we have $|S|\geqslant 1$. It is clear that $S\cup U\cup X=Hx$ and $T\cup V\cup Y=Hy$. Since $S^{-1}=T$, we have $|S|=|T|$ and so $|U|+|X|=|V|+|Y|$. Hence we can write
\begin{align*}
X=\{x_1,x_1^{-1},\dots,x_k,x_k^{-1}\},\quad &Y=\{y_1,y_1^{-1},\dots,y_\ell,y_\ell^{-1}\},\\
U=\{x_{k+1},x_{k+2},\dots,x_{k+u}\},\quad &V=\{y_{\ell+1},y_{\ell+2},\dots,y_{\ell+v}\},\\
S=\{x_{k+u+1},x_{k+u+2},\dots,x_{k+u+s}\},\quad &T=\{x_{k+u+1}^{-1},x_{k+u+2}^{-1},\dots,x_{k+u+s}^{-1}\},
\end{align*}
where $2k+u+s=2\ell+v+s=|H|$.
Without loss of generality we may assume that $k\leqslant\ell$.
If $b\leqslant 2k+1$ and $b$ is odd, then we take
\[
R=\{x_{1},x_{1}^{-1},\dots,x_{(b-1)/2},x_{(b-1)/2}^{-1}\}\cup \{y_{1},y_{1}^{-1},\dots,y_{(b-1)/2},y_{(b-1)/2}^{-1}\}\cup \{x_{k+u+1},x_{k+u+1}^{-1}\}.
\]
If $b\leqslant 2k+1$ and $b$ is even, then we take
\[
R=\{x_{1},x_{1}^{-1},\dots,x_{b/2},x_{b/2}^{-1}\}\cup \{y_{1},y_{1}^{-1},\dots,y_{b/2},y_{b/2}^{-1}\}.
\]
If $2k+1<b\leqslant2\ell+1$ and $b$ is odd, then we take
\[
R=X\cup\{x_{k+1},x_{k+2},\dots,x_{k+(b-1-2k)}\}\cup\{y_1,y_1^{-1},\dots,y_{(b-1)/2},y_{(b-1)/2}^{-1}\}\cup\{x_{k+u+1},x_{k+u+1}^{-1}\}.
\]
If $2k+1<b\leqslant2\ell+1$ and $b$ is even, then we take
\[
R=X\cup\{x_{k+1},x_{k+2},\dots,x_{k+(b-2k)}\}\cup\{y_1,y_1^{-1},\dots,y_{b/2},y_{b/2}^{-1}\}.
\]
If $2\ell+1<b\leqslant2\ell+v$, then we take
\[
R=X\cup Y\cup\{x_{k+1},x_{k+2},\dots,x_{k+(b-2k)}\}\cup\{y_{\ell+1},y_{\ell+2},\dots,y_{\ell+(b-2\ell)}\}.
\]
If $2\ell+v<b\leqslant|H|$, then we take
\[
R=X\cup Y\cup U\cup V\cup \{x_{k+u+1},x_{k+u+2},\dots,x_{b-2k-u}\}\cup \{x_{k+u+1}^{-1},x_{k+u+2}^{-1},\dots,x_{b-2k-u}^{-1}\}.
\]
In each case $R$ consists of $b$ pairwise disjoint transversals of $H$ in $HxH\cup Hx^{-1}H$, and moreover $R$ is inverse-closed.
\end{proof}

\begin{lemma}\label{lem13}
Let $G$ be a group, let $H$ be a subgroup of $G$ with $|H|=m$ odd, and let $x\in G\setminus H$. Suppose that $HxH=Hx^{-1}H$, $|HxH|/|H|=m$, and $HxH$ contains an involution. Then there exist $|H|$ pairwise disjoint inverse-closed right transversals of $H$ in $HxH$. In particular, for any integer $0\leqslant b\leqslant |H|$, there exist $b$ pairwise disjoint right transversals of $H$ in $HxH$ whose union is inverse-closed.
\end{lemma}

\begin{proof}
Since $|HxH|/|H|=m$, Lemma~\ref{lem5} implies that $\Gamma_x=K_m$. Write
\[
HxH=Hx_0\cup Hx_1\cup\dots\cup Hx_{m-1}.
\]
Consider $Hx_i$ for a fixed $i\in\{0,1,\dots,m-1\}$. Since $\Gamma_x=K_m$, for each $j\in\{0,1,\dots,m-1\}\setminus\{i\}$, there exists an element in $Hx_i$ whose inverse is in $Hx_j$.
Hence there is at most one involution in $Hx_i$ as $|Hx_i|=m$.
Moreover, since there is an involution in $HxH$, Lemma~\ref{lem2} implies that there is at least one involution in $Hx_i$.
Thus there is exactly one involution in $Hx_i$, which we denote by $x_{i,0}$.
Let $n=(m-1)/2$.
For $j=1,\dots,n$, since $\Gamma_x=K_m$, there exists $x_{i,j}\in Hx_{i+j}$ such that $x_{i,j}^{-1}\in Hx_{i-j}$, where the subscripts of $x_{i+j}$ and $x_{i-j}$ are taken modulo $m$.
Then
\[
R_i:=\{x_{i,0},x_{i,1},x_{i,1}^{-1},\dots,x_{i,n},x_{i,n}^{-1}\}
\]
is an inverse-closed right transversal of $H$ in $HxH$.
Clearly, $R_0,\dots,R_{m-1}$ are pairwise disjoint, and they form $|H|$ pairwise disjoint inverse-closed right transversals as desired.
\end{proof}

%

\section{Proofs of Theorems \ref{thm4} and \ref{thm1}}
\label{sec3}

\begin{proof}[Proof of Theorem~$\ref{thm4}$]
We only prove (a) as the proof of (b) is similar. Let
\[
G=\langle A,y\mid y^2=e,\,x^y=x^{-1}\text{ for all }x\in A\rangle
\]
be a generalized dihedral group, where $A$ is an abelian subgroup of $G$ and $x^y = y^{-1}xy$. Let $H$ be a nontrivial subgroup of $G$. The ``if" part of the statement in (a) is clearly true as a perfect code of $G$ is simply a $(0, 1)$-regular set of $G$. Now we prove the ``only if" part of the statement. Assume that $H$ is a perfect code of $G$. Let $a$ and $b$ be integers with $0\leqslant a\leqslant|H|-1$, $0\leqslant b\leqslant |H|$, and $a$ even when $|H|$ is odd. By Lemma~\ref{lem4}, to prove that $H$ is an $(a,b)$-regular set of $G$, we only need to take an inverse-closed subset of $H\setminus\{e\}$ with size $a$ and $b$ pairwise disjoint right transversals of $H$ in $G \setminus H$ which are inverse-closed in each $HxH\cup Hx^{-1}H$ for $x\in G$.

Case~1: $H\leq A$. Then $H^y=H^{-1}=H$. It follows that $H$ is normal in $G$, and so $|HzH|/|H|=|H|/|H^z\cap H|=1$ for any $z\in G$. If $HzH\neq Hz^{-1}H$, then Lemmas~\ref{lem9} asserts that there exist $b$ pairwise disjoint right transversals of $H$ in $HzH\cup Hz^{-1}H$ whose union is inverse-closed. 
Otherwise, since $H$ is a perfect code of $G$, by Theorem~\ref{thm2}, for each $x \in G\setminus H$ such that $HxH=Hx^{-1}H$ and $|H|/|H\cap H^{x}|$ is odd, there exists an involution $y'\in Hx$. This together with Lemma~\ref{lem11} implies that there exist $b$ pairwise disjoint right transversals of $H$ in $HzH\cup Hz^{-1}H$ whose union is inverse-closed.
Thus, we conclude that $H$ is an $(a,b)$-regular set of $G$.

Case~2: $H\nleq A$. Then $H$ contains an element $g\notin A $. Hence we can suppose $g=x'y$ with $x'\in A$. It follows that $g^2=x'yx'y=x'yy(y^{-1}x'y)=x'(yy)x'^y=x'x'^{-1}=e$. That is, $g$ is an involution. Thus $G=\langle A,g\rangle$, and so $H=\langle H\cap A,g\rangle$. In particular, $H=(H\cap A)\cup (H\cap A)g$ and $|H|/|H\cap A|=2$.
It is clear that $H\cap A\leq H\cap H^{z}$ for any $z\in G$.
So we have $|HzH|/|H|=|H|/|H\cap H^{z}|=1$ or $2$. Since $H$ is a perfect code of $G$, by Theorem~\ref{thm2}, for each $x\in G\setminus H$ such that $HxH=Hx^{-1}H$ and $|H|/|H\cap H^x|$ is odd, there exists an involution in $Hx$. Hence, by Lemmas~\ref{lem9}, \ref{lem11} and \ref{lem12}, $H$ is an $(a,b)$-regular set of $G$.
\end{proof}

%

\begin{lemma}\label{lem8}
Let $G$ be a group of order $4p$, where $p$ is a prime, let $H$ be a nontrivial subgroup of $G$, and let $x\in G$. If $HxH\neq Hx^{-1}H$, then for any integer $0\leqslant b\leqslant |H|$, there exist $b$ pairwise disjoint right transversals of $H$ in $HxH\cup Hx^{-1}H$ whose union is inverse-closed.
\end{lemma}

\begin{proof}
Since $|G|=4p$ with $p$ prime, we have $|H|=r$ or $2r$ with $r\in\{2,p\}$. If $|H|=r$, then Lemma~\ref{lem10} yields the desired result.
Thus assume that $|H|=2r$ in the rest of the proof. Then $|HxH|/|H|=|H|/|H\cap H^x|=1,r$ or $2r$. If $|HxH|/|H|=1$ or $2$, then Lemma~\ref{lem9} yields the desired result.

Next assume that $|HxH|/|H|=r$. Then by Lemma~\ref{lem5} we have $\Gamma_x=K_{r,r}$. By K\"{o}nig's 1-factorization theorem \cite{Konig1976}, $\Gamma_x$ can be decomposed into $r$ edge-disjoint perfect matchings. These matchings induce $r$ pairwise disjoint inverse-closed right transversals of $H$ in $HxH\cup Hx^{-1}H$, denoted by $R_1,R_2,\dots,R_r$, where $R_i^{-1}=R_i$ and $R_i\cap R_j=\emptyset$ for all distinct $i,j\in\{1,2,\dots,r\}$. Write $HxH=\{Hx_1,Hx_2,\dots,Hx_r\},Hx^{-1}H=\{Hx_{r+1},Hx_{r+2},\dots,Hx_{2r}\}$ and
\[M=(HxH\cup Hx^{-1}H)\setminus (R_1\cup R_2\cup\dots\cup R_r).
\]
Then for each $i\in \{1,2,\dots,2r\}$, since $|Hx_i|=|H|=2r$ and
\begin{align*}
M\cap Hx_i&=\big((Hx_1\cup Hx_2\cup\dots\cup Hx_{2r})\setminus (R_1\cup R_2\cup\dots\cup R_r)\big)\cap (Hx_i)\\
&=\big((Hx_1\cup Hx_2\cup\dots\cup Hx_{2r})\cap (Hx_i)\big)\setminus \big((R_1\cup R_2\cup\dots\cup R_r)\cap Hx_i\big)\\
&=(Hx_i)\setminus \big((R_1\cup R_2\cup\dots\cup R_r)\cap Hx_i\big),
\end{align*}
we have
\[
|M\cap Hx_i|=|Hx_i|-|(R_1\cup R_2\cup\dots\cup R_r)\cap Hx_i|=2r-r=r.
\]
Hence $M$ is the union of $r$ pairwise disjoint right transversals of $H$ in $HxH\cup Hx^{-1}H$. Moreover, since both $HxH\cup Hx^{-1}H$ and $R_1\cup R_2\cup\dots\cup R_r$ are inverse-closed, $M$ is inverse-closed as well. If $0\leqslant b\leqslant r$, then we take
\[
R=R_1\cup R_2\cup\dots\cup R_b.
\]
If $r< b\leqslant 2r$, then we take
\[
R=M\cup R_1\cup R_2\cup\dots\cup R_{b-r}.
\]
Hence, for $0\leqslant b\leqslant 2r$, $R$ consists of $b$ pairwise disjoint transversals of $H$ in $HxH\cup Hx^{-1}H$, and moreover $R$ is inverse-closed.

Now assume that $|HxH|/|H|=2r$. Then $\Gamma_x=K_{2r,2r}$ by Lemma~\ref{lem5}. Again, by K\"{o}nig's 1-factorization theorem \cite{Konig1976}, $\Gamma_x$ can be decomposed into $2r$ edge-disjoint perfect matchings. These $2r$ perfect matchings give rise to $2r$ pairwise disjoint right transversals of $H$ in $HxH\cup Hx^{-1}H$, and each of these right transversals is inverse-closed.
\end{proof}

%
%

\begin{lemma}\label{lem14}
Let $G$ be a group of order $4p$, where $p$ is a prime, let $H$ be a nontrivial subgroup of $G$, and let $x\in G\setminus H$ be such that $HxH=Hx^{-1}H$. Suppose that $HxH$ contains an involution when $|HxH|/|H|$ is odd. Then for any integer $0\leqslant b\leqslant |H|$, there exist $b$ pairwise disjoint right transversals of $H$ in $HxH$ whose union is inverse-closed.
\end{lemma}

\begin{proof}
If $|HxH/H|=1$ or $2$, then the result follows from Lemmas~\ref{lem11} and \ref{lem12}. Thus assume that $|HxH/H|\geqslant3$ in the rest of the proof. Then $|G|\neq8$, and so $p$ is odd. Hence $|H|=2$, $4$, $p$ or $2p$. If $|H|=2$, then $|HxH|/|H|=|H|/|H\cap H^x|\leqslant2$, a contradiction. If $|H|=2p$, then $|HxH|/|H|\leqslant|G|/|H|=2$, again a contradiction. If $|H|=p$, then the result follows from Lemmas~\ref{lem11} and \ref{lem13}.

It remains to consider the case $|H|=4$. Since $|HxH|/|H|\geqslant3$, we have $|HxH|/|H|=4$. So $\Gamma_x$ is the complete graph $K_4$. Consequently, there are three pairwise disjoint perfect matchings in $\Gamma_x$ which induce three pairwise disjoint inverse-closed transversals of $H$ in $HxH\cup Hx^{-1}H$. 
\end{proof}

\begin{proof}[Proof of Theorem~$\ref{thm1}$]
We only prove (a) as the proof of (b) is similar. Let $G$ be a group of order $4p$ or $pq$ for some primes $p$ and $q$, and let $H$ be a nontrivial subgroup of $G$. Since a perfect code of $G$ is exactly a $(0, 1)$-regular set of $G$, the ``if" part of the statement in (a) is clearly true.

Now we prove the ``only if" part of the statement in (a). Assume that $H$ is a perfect code of $G$. Let $a$ and $b$ be integers with $0\leqslant a\leqslant|H|-1$, $0\leqslant b\leqslant |H|$, and $a$ even when $|H|$ is odd. By Lemma~\ref{lem4}, to prove that $H$ is an $(a,b)$-regular set of $G$, it suffices to take an inverse-closed subset of $H\setminus\{e\}$ with size $a$ and $b$ pairwise disjoint right transversals of $H$ in $G \setminus H$ which are inverse-closed in each $HxH\cup Hx^{-1}H$ for $x\in G$.

Since $H$ is a perfect code of $G$, by Theorem~\ref{thm2}, for each $x\in G\setminus H$ such that $HxH=Hx^{-1}H$ and $|H|/|H\cap H^x|$ is odd, there exists an involution $y\in Hx$. Hence, if $|G|=4p$ for some prime $p$, then $H$ is an $(a,b)$-regular set of $G$ by Lemmas~\ref{lem8} and \ref{lem14}. The rest of the proof handles the case $|G|=pq$.

Case~1: Both $p$ and $q$ are odd primes. Suppose that there exists some $x\in G\setminus H$ such that $HxH=Hx^{-1}H$. Since $H$ is a perfect code of $G$, Lemma~\ref{thm2} implies that there exists some $z\in G\setminus H$ with $z^2\in H$ such that $Hz$ contains an involution, which contradicts the assumption that $|G| = pq$ is odd. Hence $HxH\neq Hx^{-1}H$ for any $x\in G\setminus H$. Since $|G|=pq$ with both $p$ and $q$ odd primes, we have $|H|=r$ for $r \in \{p, q\}$. Since $|HxH|/|H|=|H|/|H\cap H^x|$, it follows that $|HxH|/|H|=1$ or $r$. Thus, by Lemmas~\ref{lem9} and \ref{lem10}, $H$ is an $(a,b)$-regular set of $G$.

Case~2: One of $p$ and $q$ is even, say, $p=2$. Then $|H|=2$ or $q$. First assume that $HxH\neq Hx^{-1}H$.
If $|H|=2$, then $|HxH|/|H|=|H|/|H\cap H^x|\leqslant2$, and by Lemma~\ref{lem9}, $H$ is an $(a,b)$-regular set of $G$. If $|H|=q$, then by Lemma~\ref{lem10}, $H$ is an $(a,b)$-regular set of $G$. Next assume that $HxH=Hx^{-1}H$. If $|H|=2$, then $|HxH|/|H|=|H|/|H\cap H^x|\leqslant2$, and by Lemmas~\ref{lem11} and \ref{lem12}, $H$ is an $(a,b)$-regular set of $G$. If $|H|=q$, then $|HxH|/|H|=|H|/|H\cap H^x|$ is $1$ or $q$, and hence by Lemmas~\ref{lem11} and \ref{lem13}, $H$ is an $(a,b)$-regular set of $G$.
\end{proof}

\medskip
\noindent \textbf{Acknowledgements}
\medskip

We are grateful to the anonymous referees for their careful reading and helpful suggestions. The third author was supported by the Research Grant Support Scheme of The University of Melbourne.

\end{document}